\documentclass[12pt,oneside,english]{amsart}
\usepackage[T1]{fontenc}
\usepackage[latin9]{inputenc}
\usepackage{geometry}
\usepackage{amsthm}
\usepackage{amssymb}

\makeatletter
\numberwithin{equation}{section}
\numberwithin{figure}{section}
\theoremstyle{plain}
\newtheorem{thm}{\protect\theoremname}
\theoremstyle{plain}
\newtheorem{lem}{\protect\lemmaname}

\makeatother

\usepackage{babel}
\providecommand{\lemmaname}{Lemma}
\providecommand{\theoremname}{Theorem}

\begin{document}
\title{Upper Bounds For Families Without Weak Delta-Systems}
\author{Eric Naslund}
\date{\today}
\email{naslund.math@gmail.com}
\begin{abstract}
For $k\geq3$, a collection of $k$ sets is said to form a \emph{weak
$\Delta$-system} if the intersection of any two sets from the collection
has the same size. Erd\H{o}s and Szemer\'{e}di asked about the size
of the largest family $\mathcal{F}$ of subsets of $\{1,\dots,n\}$
that does not contain a weak $\Delta$-system. In this note we improve
upon the best upper bound due to the author and Sawin, and show that
\[
|\mathcal{F}|\leq\left(\frac{2}{3}\Theta(C)+o(1)\right)^{n}
\]
 where $\Theta(C)$ is the capset capacity. In particular, this shows
that
\[
|\mathcal{F}|\leq(1.8367\ldots+o(1))^{n}.
\]
\end{abstract}

\maketitle

\section{Introduction}

A collection of $k$ sets for $k\geq3$ is said to form a \emph{$k$-sunflower,
}or a $\Delta$\emph{-system,} if the intersection of any two sets
from the collection is the same. This notion was introduced in 1960
by Erd\H{o}s and Rado \cite{ErdosRado1960SunflowerSystems}, and
they famously conjectured that for $k\geq3$, there exists $C_{k}>0$
depending on $k$, such that any $k$-sunflower-free family $\mathcal{F}$
of sets of size $m$ satisfies 
\[
|\mathcal{F}|\leq C_{k}^{m}.
\]
This was one of Erd\H{o}s's favorite problems, for which he offered
$\$1000$ \cite[Problem 90]{Chung1997OpenProblemsOfPaulErdosGraphTheory},
and while still out of reach, progress was made recently by Alweiss,
Lovett, Wu and Zhang \cite{AlweissLovettWuZhang2021ImprovedSunflowers}
(see also \cite{Rao2020CodingForSunflowers,Tao2020BlogSunflowers,BellTolsonChueluecha2021NoteOnSunflowerFurtherImprovementsRao,Rao2023SunflowerExposition}).

In 1974, Erd\H{o}s, E. Milner, and Rado \cite{ErdosMilnerRado1974WeakDeltaSystems}
introduced the related notion of a weak $\Delta$-system (see \cite{Kostochka1998ExtremalProblemsDeltaSystems}
for a survey). For $k\geq3$, a collection of $k$ sets is said to
form a \emph{weak $\Delta$}-\emph{system of size $k$ }if the intersection
of any two sets from the collection has the same size. Let $G_{k}(n)$
denote the size of the largest family of subsets of $\{1,\dots,n\}$
that does not contain a weak $\Delta$-system of size $k$. Erd\H{o}s
and Szemer\'{e}di asked about the growth rate of $G_{k}(n)$ \cite{ErdosSzemeredi1978UpperBoundsForSunflowerFreeSets}
(see also \cite[Problem 94]{Chung1997OpenProblemsOfPaulErdosGraphTheory}),
and in their paper they gave the lower bound 
\[
G_{k}(n)\geq n^{\frac{\log n}{4\log\log n}}.
\]
The current best lower bound, due to Kostochka and R\"{o}dl \cite{KostochkaRodl1998LowerBoundWeakDealtaSystems}
(which improved upon \cite{RodlThoma1997WeakDeltaSystemsLowerBound}),
is
\[
G_{k}(n)\geq k^{c(n\log n)^{\frac{1}{3}}}.
\]
Frankl and R\"{o}dl \cite{FranklRodl1987FranklRodlTheorem} resolved
a conjecture of Erd\H{o}s and Szemer\'{e}di, and proved that for
every $k$, there exists $\varepsilon_{k}>0$ such that $G_{k}(n)<(2-\varepsilon_{k})^{n}$.
In the case $k=3$, the stronger upper bound 
\begin{equation}
G_{3}(n)\leq\left(\frac{3}{2^{2/3}}+o(1)\right)^{n}=\left(1.889881\ldots+o(1)\right)^{n}\label{eq:sunflower_free_upper_bound}
\end{equation}
is a consequence of the upper bound for $3$-sunflower-free families
due to the author and Sawin \cite{NaslundSawin2017Sunflower}. That
result was proven using the \emph{slice-rank method}, which was introduced
by Tao \cite{TaosBlogCapsets} following the polynomial method breakthrough
of Croot, Lev, and Pach \cite{CrootLevPachZ4}, and Ellenberg and
Gijswijt \cite{EllenbergGijswijtCapsets}. Throughout, we refer to
a $3$-sunflower-free family simply as \emph{sunflower-free }since
it contains no $k$-sunflower for any $k$. The bound for sunflower-free
sets achieved in \cite{NaslundSawin2017Sunflower} cannot be improved
without a substantial change in approach, since the result also applies
to multicolored sunflower-free sets (see \cite{BlasiakChurchCohnGrochowNaslundSawinUmans2016MatrixMultiplication}
for a definition). The multicolored lower bounds from the work of
Kleinberg, Sawin, and Speyer, \cite{KleinbergSawinSpeyer2017TheGrowthRateOfTriColoredSumFreeSets},
and Pebody \cite{Pebody2017ProofOfKleinbergSawinSpeyerConjecture},
imply that for multicolored sunflower-free sets, the bound in equation
(\ref{eq:sunflower_free_upper_bound}) is optimal up to sub-exponential
factors. Note that if a family does not contain a weak $\Delta$-system
for $k=3$, then it does not contain a weak $\Delta$-system for any
$k$.

In this paper we relate the size of the largest family that does not
contain a weak $\Delta$-system in $\{0,1\}^{n}$ to the size of the
largest capset in $\mathbb{F}_{3}^{n}$, and improve upon the upper
bound for $G_{3}(n)$. A set $A\subset\mathbb{F}_{3}^{n}$ is called
a \emph{capset} if there is no triple $x,y,z\in A$, not all equal,
such that $x+y+z=0\pmod{3}$. Equivalently, $A$ is a capset if there
does not exist a triple $x,y,z$, not all equal, such that for every
coordinate $i$, 
\[
\{x_{i},y_{i},z_{i}\}\in\{\{0,1,2\},\{0,0,0\},\{1,1,1\},\{2,2,2\}\}.
\]
Let $C_{n}$ denote the size of largest capset in $\mathbb{F}_{3}^{n}$,
and define the \emph{capset capacity, }$\Theta(C)$, to be 
\[
\Theta(C)=\limsup_{n\rightarrow\infty}\left(C_{n}\right)^{\frac{1}{n}}.
\]
In particular, Ellenberg and Gijswijt proved that 
\begin{equation}
\Theta(C)\leq\min_{0<t<1}t^{-\frac{2}{3}}(1+t+t^{2})=\frac{3}{8}\sqrt[3]{207+33\sqrt{33}}=2.7551046\ldots.\label{eq:ellenberg_gijswijt_bound}
\end{equation}
The notation $\Theta(C)$ is used since this quantity is precisely
the Shannon Capacity of the hypergraph with three elements and one
edge, see \cite{ChristandlFawziOmarTa2022LargerCornerFreeSets} for
details on this notation. Our main result is:
\begin{thm}
\label{thm:main_weak_delta_system_thm}Let $X$ be a set of size $|X|=n$,
and let $\mathcal{F}$ be a collection of subsets of $X$, and suppose
that $\mathcal{F}$ does not contain a weak $\Delta$-system. Then
\[
|\mathcal{F}|\leq\left(\frac{2}{3}\Theta(C)+o(1)\right)^{n},
\]
where $\Theta(C)$ is the capset capacity. In particular due to (\ref{eq:ellenberg_gijswijt_bound})
we have that 
\begin{equation}
|\mathcal{F}|\leq\left(\frac{1}{4}\sqrt[3]{207+33\sqrt{33}}+o(1)\right)^{n}=(1.8367\ldots+o(1))^{n}.\label{eq:weak_delta_numerical_upper_bound}
\end{equation}
\end{thm}
Equivalently, Theorem \ref{thm:main_weak_delta_system_thm} states
that 
\[
G_{3}(n)\leq\left(\frac{2}{3}\Theta(C)+o(1)\right)^{n}.
\]
To prove this, we examine sets without non-trivial equilateral triangles
in $\{0,1\}^{n}$, where an \emph{equilateral triangle} is a triple
$x,y,z\in\mathbb{R}^{n}$ such that $\|x-y\|=\|y-z\|=\|z-x\|$, and
it is said to be trivial\emph{ }if $x=y=z$. In the next section,
we prove the following upper bound:
\begin{thm}
\label{thm:equilateral_triangles_upper}Let $A\subset\{0,1\}^{n}$
that does not contain a non-trivial equilateral triangle. Then $|A|\leq(\frac{2}{3}\Theta(C)+o(1))^{n}$
where $\Theta(C)$ is the capset capacity.
\end{thm}
Since we are working in $\{0,1\}^{n}$, coordinate-wise distances
are either $0$ or $1$, and so the above result holds for any $L^{p}$
norm. Let us begin by deducing Theorem \ref{thm:main_weak_delta_system_thm}
from Theorem \ref{thm:equilateral_triangles_upper}.
\begin{proof}[Proof of Theorem \ref{thm:main_weak_delta_system_thm} assuming Theorem
\ref{thm:equilateral_triangles_upper}]
 Let $\mathcal{F}$ be a family of subsets of $\{1,\dots,n\}$, and
suppose that $\mathcal{F}$ does not contain a weak $\Delta$-system.
Every subset $A\subset\{1,2,\dots.n\}$ corresponds to a vector $x\in\{0,1\}^{n}$
where $x_{i}=1$ if and only if $i\in A$, and so our family $\mathcal{F}$
corresponds to a set $A\subset\{0,1\}^{n}$. In this setting, three
vectors $x,y,z\in\{0,1\}^{n}$ form a weak $\Delta$-system if and
only if $\langle x,y\rangle=\langle y,z\rangle=\langle z,x\rangle$.
For $x\in\{0,1\}^{n}$, the \emph{weight }of $x$ is defined to be
the number of non-zero entries. If $x,y,z$ all have the same weight
$w$, then $\|x\|_{2}^{2}=\|y\|_{2}^{2}=\|z\|_{2}^{2}$, and so $x,y,z$
form a weak $\Delta$-system if and only if $\|x-y\|_{2}^{2}=\|y-z\|_{2}^{2}=\|z-x\|^{2}$,
that is, if and only if $x,y,z$ form an equilateral triangle. Let
$A_{w}$ denote the elements of $A$ of weight $w$. Since $\sum_{w=0}^{n}|A_{w}|=|A|$,
there must exist $w$ such that $|A_{w}|\geq\frac{|A|}{n+1}$. Then
$A_{w}$ is a set that does not contain an equilateral triangle, and
so by Theorem \ref{thm:equilateral_triangles_upper}, $|A_{w}|\leq(\frac{2}{3}\Theta(C))^{n}$,
and the proof is complete.
\end{proof}

\section{Subsets of $\{0,1\}^{n}$ Avoiding Equilateral Triangles}

To prove Theorem \ref{thm:equilateral_triangles_upper}, we first
prove a lemma that allows us to upper bound the density of the largest
set without equilateral triangles in $\{0,1\}^{n}$ by the the relative
density of the largest set without equilateral triangles inside \emph{any
subset} of $\{0,1\}^{n}$. Then we define a mapping that allows us
to use the Ellenberg-Gijswijt capset bound to upper bound the density
of the largest set without equilateral triangles among the elements
of weight $w$.

For any $x\in\{0,1\}^{n}$ define the map $f_{x}:\{0,1\}^{n}\rightarrow\{0,1\}^{n}$
by 
\[
f_{x}(y)=x+y\pmod{2}.
\]

\begin{lem}
\label{lem:function_lemma}For any $x\in\{0,1\}^{n}$, $f_{x}$ is
an isometry. That is, for any $y,z\in\{0,1\}^{n}$ we have that $\|y-z\|=\|f_{x}(y)-f_{x}(z)\|$.
\end{lem}
\begin{proof}
Given $x,y,z$, examine coordinate by coordinate. If $y_{i}=z_{i}$,
then $x_{i}+y_{i}=x_{i}+z_{i}\pmod{2}$, and so the distance is still
$0$. If $y_{i}\neq z_{i}$, then $x_{i}+y_{i}\neq x_{i}+z_{i}\pmod{2}$,
and so once again the distance is preserved.
\end{proof}
For any set $B\subset\{0,1\}^{n}$, let $w_{\Delta}(B)$ denote the
size of the largest subset of $B$ that does not contain an equilateral
triangle, and define 
\[
\delta_{\Delta}(B)=\frac{w_{\Delta}(B)}{|B|}.
\]
Note that if $B$ does not contain any equilateral triangles, then
$\delta_{\Delta}(B)=1$.
\begin{lem}
\label{lem:relative_density_lemma}Let $A\subset\{0,1\}^{n}$ that
does not contain an equilateral triangle. Then 
\[
|A|\leq2^{n}\min_{B\subset\{0,1\}^{n}}\delta_{\Delta}(B).
\]
\end{lem}
\begin{proof}
Let $A,B\subset\{0,1\}^{n}$ be given, and suppose that $A$ does
not contain an equilateral triangle. For every element $x\in\{0,1\}^{n}$
consider $f_{x}(A)\cap B$. For each pair of elements, $a\in A,b\in B$,
there is one and only one element $x\in\{0,1\}^{n}$ such that $f_{x}(a)=b$,
which implies that 
\[
\sum_{x\in\{0,1\}^{n}}|f_{x}(A)\cap B|=|A||B|,
\]
and hence there exists $x\in\{0,1\}^{n}$ such that 
\[
\frac{|A|}{2^{n}}\leq\frac{|f_{x}(A)\cap B|}{|B|}.
\]
Since $A$ does not contain an equilateral triangle, by Lemma \ref{lem:function_lemma},
neither does $f_{x}(A)$. Hence
\[
\frac{|f_{x}(A)\cap B|}{|B|}\leq\delta_{\Delta}(B),
\]
by definition of $\delta_{\Delta}$, and the lemma follows.
\end{proof}
We say that a subset of $\{0,1\}^{n}$ is \emph{sunflower-free} if
it does not contain three elements $x,y,z$, not all equal, such that
$\{x_{i},y_{i},z_{i}\}\in\{\{0,0,0\},\{1,1,1\},\{0,0,1\}\}$ for every
$i$. Note that in this definition of sunflower-free, we do allow
triples where two of the three are equal, and so for example $A=\{(0,1),(1,1)\}$
is not sunflower-free since $(0,1),(0,1),(1,1)$ form a sunflower.
For a set of vectors of a fixed weight, this definition of sunflower-free
is the same if the three vectors are required to be distinct. Consider
the map $F:\mathbb{F}_{3}^{n}\rightarrow\{0,1\}^{n}$ defined coordinate-wise
by $F_{i}(0)=0$, $F_{i}(1)=1$ and $F_{i}(2)=0$ for each $i$.
\begin{lem}
\label{lem:inverse_map_sunflower_capset}Let $A\subset\{0,1\}^{n}$
be a sunflower-free set. Then $F^{-1}(A)$ is a capset.
\end{lem}
\begin{proof}
We will show that if $x,y,z\in\mathbb{F}_{3}^{n}$ are not all equal,
and if $F(x),F(y),F(z)$ is not a sunflower, then $x+y+z\neq0\pmod{3}$.
If $F(x)=F(y)=F(z)$, then we cannot have $x+y+z=0$ unless $x=y=z$
since $\{0,2\}^{n}$ is a capset in $\mathbb{F}_{3}^{n}$. If $F(x),F(y),F(z)$
are not all equal, and do not form a sunflower, then there exists
a coordinate $i$ such that $\{F(x_{i}),F(y_{i}),F(z_{i})\}=\{0,1,1\}$.
This implies that $\{x_{i},y_{i},z_{i}\}=\{*,1,1\}$ where $*$ is
either a $0$ or a $2$, and in either case this guarantees that $x+y+z\neq0$.
\end{proof}
Let $B_{k}\subset\{0,1\}^{n}$ denote the set of vectors of weight
$k$. 
\begin{thm}
\label{thm:special_sunflower_upper_bound} If $A\subset B_{k}$ does
not contain a sunflower, we have that 
\[
|A|\leq\frac{\Theta(C)^{n}}{2^{n-k}}
\]
where $\Theta(C)$ is the capset capacity.
\end{thm}
\begin{proof}
Let $A\subset B_{k}$ be a sunflower-free set. Each element in $A$
has $n-k$, zeros, and so $|F^{-1}(A)|=|A|\cdot2^{n-k}$. The result
follows since Lemma \ref{lem:inverse_map_sunflower_capset} implies
that $|F^{-1}(A)|\leq\Theta(C)^{n}$ since $\Theta(C)^{n}$ bounds
from above the size of the largest capset of size $n$.
\end{proof}
Putting this all together, we prove Theorem \ref{thm:equilateral_triangles_upper}.
\begin{proof}[Proof of Theorem \ref{thm:equilateral_triangles_upper}]
 A set without equilateral triangles in $\mathbb{R}^{n}$ gives rise
to such a set in $\mathbb{R}^{n+m}$ by appending $m$ $0$'s to each
vector, so we may suppose that $3|n$ which can only impact the bound
by at most a factor of $4$. Among the vectors of weight $n/3$, every
sunflower is an equilateral triangle. Since 
\[
|B_{n/3}|=\binom{n}{n/3}=\left(\frac{3}{2^{2/3}}+o(1)\right)^{n},
\]
due to Stirling's approximation, Theorem \ref{thm:special_sunflower_upper_bound}
implies that for $B_{n/3}$,
\[
\delta_{\Delta}(B_{n/3})\leq\left(\frac{\Theta(C)}{2^{2/3}}\right)^{n}\cdot\left(\frac{3}{2^{2/3}}+o(1)\right)^{-n}=\left(\frac{\Theta(C)}{3}+o(1)\right)^{n}.
\]
Lemma \ref{lem:relative_density_lemma} implies that for any $A\subset\{0,1\}^{n}$
that does not contain an equilateral triangle, we have 
\[
|A|\leq2^{n}\delta_{\Delta}(B_{n/3})\leq\left(\frac{2\Theta(C)}{3}+o(1)\right)^{n},
\]
and the result follows. 
\end{proof}

\specialsection*{Acknowledgements}

I would like to thank Lisa Sauermann for her many helpful comments,
and for pointing out an error in the original version of this paper.
I would also like to thank the anonymous referees for their valuable
feedback.

\bibliographystyle{plain}

\end{document}